\documentclass{amsart}
  \usepackage{amssymb, array, verbatim, amscd, amsmath, amsfonts,amsthm, color}

\theoremstyle{plain}
\newtheorem{theorem}{Theorem}[section]

\newtheorem{corollary}[theorem]{Corollary}
\newtheorem{lemma}[theorem]{Lemma}

\theoremstyle{definition}

\newtheorem{example}[theorem]{Example}

\newtheorem{remark}[theorem]{Remark}

\newcommand{\la}{\langle}
\newcommand{\ra}{\rangle}

\newcommand{\bbZ}{\mathbb Z}
\newcommand{\Z}{\mathbb Z}
\newcommand{\bbN}{\mathbb N}

\newcommand{\calP}{\mathcal {P}}

\newcommand{\calS}{\mathcal {S}}

\newcommand{\calQ}{\mathcal Q}
\newcommand{\calW}{\mathcal W}

\newcommand{\res}{\mathit{res}}

 \newcommand{\End}  {\operatorname{End}}
\newcommand{\Hom}{\operatorname{Hom}}
\newcommand{\Mon}{\operatorname{Mon}}

\newcommand{\im}{\operatorname{im}}

\newcommand{\ord} {\operatorname{ord}}
\newcommand{\Sub} {\operatorname{Sub}}

\begin{document}

\title{Subgroups which admit extensions of homomorphisms}
\author{Simion Breaz, Grigore C\u{a}lug\u{a}reanu and Phill Schultz}\thanks{S. Breaz is supported by the
CNCS-UEFISCDI grant PN-II-RU-TE-2011-3-0065}
\date{\today}

\address{(Breaz) "Babe\c s-Bolyai" University, Faculty of Mathematics
and Computer Science, Str. Mihail Kog\u alniceanu 1, 400084
Cluj-Napoca, Romania} \email{bodo@math.ubbcluj.ro}

\address{(C\u alug\u areanu) "Babe\c s-Bolyai" University, Faculty of Mathematics
and Computer Science, Str. Mihail Kog\u alniceanu 1, 400084
Cluj-Napoca, Romania} \email{calu@math.ubbcluj.ro}

 \address{(Schultz) School of Mathematics and Statistics, The University of Western Australia,
Nedlands, 6009, Australia} \email{phill.schultz@uwa.edu.au}


 \begin{abstract}
We classify by numerical invariants the finite subgroups $H$ of a
primary abelian group $G$ for which  every homomorphism or
monomorphism of $H$ into $G$, or every   endomorphism of $H$,
extends to an endomorphism of $G$. We apply these results to show
that for finitely generated subgroups of general abelian groups,
the extendibility of monomorphisms implies the extendibility of
all homomorphisms.
\end{abstract}


\maketitle

\section{Introduction}

 We begin by characterizing in module theoretic terms
the extension properties described above.  The problem of
characterizing the subgroups satisfying these properties  forms
part of the more general question of characterizing special
classes of submodules of a module. For example,
  Birkhoff observed in \cite{B34} that although for
some rings $R$ (and in particular for $R=\Z$) finitely generated
modules can be completely characterized using  numerical
invariants, in general it is difficult to describe even cyclic
modules as submodules of a given module. A complex study in this
direction was initiated by   Ringel and  Schmidmeier \cite{RS08}
for artinian algebras, and it was continued by several authors who
show that in many cases the category of submodules is \lq wild\rq\
(cf. the introduction of \cite{RS06}). For the case of abelian
groups, we mention the studies realized in \cite{A00} and
\cite{RW99}.

 Let  $R$ be a unital ring and $M$   an $R$--module, and let $\Sub(M)$ be the set of all
submodules of $M$.

If $N\in\Sub(M)$, and $\iota:\, N\to M$ is the inclusion map, by
applying
  the contravariant functors $\Hom(-,M)$ and $\Hom(-,N)$, we obtain a commutative diagram:
\[\begin{CD} \Hom(M,N)@> \iota_M^*>>\End(M)\\
@V \res_N VV @VV \res_M V\\
\End(N) @>>\iota_N^*> \Hom(N,M)\end{CD},\] \noindent where
$\iota_X^*:\,\Hom(X,N) \to
\Hom(X,M)$ are the induced inclusion maps and
$\res_X:\,\Hom(M,X)\to\Hom(N,X)$ the induced restriction maps.

 Several module theoretic properties that appear in the literature in other guises can be described in terms of this diagram:    $\iota_M^*$ is an isomorphism if and only if
$N=M$; if $\res_M$ is an isomorphism then $M$ is a
 localization of $N$ (\cite{Du04}); $\res_N$ is
an isomorphism   if and only if $N$ is a direct summand with a unique
complement;  $\res_M$ factors through $\iota_N^*$  if and only if $N$ is fully invariant in $M$; for a given $N$, $\res_M$ and $\im^*_N$ have the same image for all $M$ if and only if $N$ is rigid \cite{DG96}.

Now consider, for a given module $M$,   the following sets of
submodules:
\begin{align*} \calS (M)&=\{ N\in \Sub(M):\, \res_N \textrm{ is epic}\};\\
\calQ (M)&=\{ N\in \Sub(M):\,  \res_M \textrm{ is epic}\};\\
 \calW (M)&=\{ N\in \Sub(M):\,  \im(\res_M)=\im(\iota^*_N)\};\\
 \calP(M)&=\{N\in\Sub(M): \im(\res_M)\supseteq\Mon(N,M)\}
\end{align*}
where $\Mon(N,M)$ is the set of monomorphisms of $N$ into $M$.

It is easy to see that $\calS(M)$ is the set of all direct
summands of $M$, $\calQ(M)$ is the set of all submodules $N$ of
$M$ such that all homomorphisms of $N$ into $M$ can be extended to
endomorphisms of $M$,   $\calW(M)$ is the class of all submodules
$N$ of $M$ such that all endomorphisms of $N$ can be extended to
endomorphisms of $M$.  and $\calP(M)$ is the class of all
submodules $N$ of $M$ for which all monomorphisms of $N$ into $M$
can be extended to endomorphisms of $M$.

Therefore, if one of these classes coincides with the set of all
submodules  of $M$, i.e., $\mathcal{X}(M)=\Sub(M)$ for
$\mathcal{X}=\mathcal{S}$, $\calQ$, $\calW$, or $\calP$ then $M$
is semi-simple, respectively quasi-injective (\cite{Fa67}),
weakly-injective (\cite{mis}) or pseudo-injective (\cite{JS67}).
Using  known results about the structure of these modules, it is
easy to see that in general $\calS(M)\subseteq \calQ(M) \subseteq
\calP(M)\subseteq \Sub(M)$ and all inclusions
  can be strict.
With one exception (the inclusion $\calQ(M)\subseteq \calP(M)$)
 this strictness can be demonstrated in the class of   abelian groups, using structure theorems in \cite{kil} and \cite{mis}. For
$\calQ(M)\subsetneqq \calP(M)$, there are pseudo-injective modules
which are not quasi-injective (see \cite{JS67} or \cite{T75}).
However, it is proved in \cite{sin} that over principal ideal
domains quasi-injective and pseudo-injective modules coincide,
 and we do not know if there exists an abelian group
$G$ such that $\calQ(G)\neq \calP(G)$.

Recently Er,  Singh and Srivastava \cite{ESS13} showed that
pseudo--injective modules are precisely those modules which are
invariant under automorphisms of their injective hulls.   This
description should be compared with the well known
characterization of quasi--injective modules as those modules
which are fully invariant, i.e., invariant under endomorphisms of
their injective hulls.

The object of the paper is to characterize finitely generated
subgroups which lie in these classes for the case of primary
abelian groups. Let $G$ be a $p$-group. Cyclic subgroups in
$\calQ(G)$ are described in Theorem \ref{3.14}, while finitely
generated subgroups in $\calQ(G)$ are characterized in Theorem
\ref{finite-Q(G)}.  Using these results we prove in Theorem
\ref{Pf=Qf} that for all abelian groups $G$ (not only for
quasi-injective or pseudo-injective abelian groups) we have
$\calQ_f(G)=\calP_f(G)$, where $\mathcal{X}_f(G)$ denotes the set
of all finitely generated subgroups in $\mathcal{X}(G)$. Theorem
\ref{finite-W(G)} gives us information about finitely generated
subgroups in $\calW(G)$.

All groups in this paper are abelian. Unless specifically noted we
use the standard notation of \cite{fuc1} and \cite{fuc2}.

\textsl{With the exception of the final section, in the rest of
this paper we assume that   $p$ is a fixed prime and $G$   a
$p$-group.}

If $G$ is bounded, we denote by $\exp(G)$ the least positive
integer $k$ such that $p^k G=0$. If $G$ is not bounded then
$\exp(G)=\infty$.  The exponent of an element $x\in G$,  denotes
the positive integer $\exp(x)$ such that the order of $x$ is
$p^{\exp(x)}$. If   $x\in G$, then $h(x)$ denotes the  height of
$x$.   A    group $G$ is \textit{homogeneous}  if it is a direct sum of isomorphic   quasi-cyclic
groups, i.e. $G\cong \bbZ(p^k)^{(\lambda)}$ where $k$ is a
positive integer  or $\infty$ and $\lambda$ is a
cardinal.

It is well known that the semi--simple  groups are the direct sums
  of elementary  groups.
Moreover, quasi-injective  (primary) groups are
exactly the homogeneous groups.

{  \section{Cyclic subgroups in $\calQ(G)$}\label{sect-cyclic}

In this section, we   consider the question:
 which cyclic subgroups of $G$ are in $\calQ(G)$ or in $\calP(G)$?

If $G$ is divisible, then for all $x\in G,\ \la x\ra\in\calQ(G)$,
so we need consider only non--divisible $G$. We recall that
 $G^1=\bigcap_{n>0}p^nG$ denotes {\sl the first Ulm
subgroup} of $G$.

\begin{lemma}\label{cyclic2}
Let $G$ be   non--divisible   and $\langle
x\rangle\in\calP(G)$.

\begin{enumerate}
\item If $G$ has an unbounded basic subgroup then $\langle
x\rangle\cap G^1=\{0\}$.

\item If $G=B\oplus D$ with $B$ bounded and $D$ divisible then either
\begin{enumerate} \item $\exp(x)\leq\exp(B)$ and $\langle x\rangle\cap D=0$ or
\item $\exp(x)>\exp(B)$ and $x=b+d$ with $b\in B$, $d\in D$ such that $\exp(B)=\exp(b)$
and $\exp(d)>\exp(B)$.
\end{enumerate}
\end{enumerate}
\end{lemma}

\begin{proof}

(1) If $G$ has an unbounded basic subgroup then $G$ has a
cyclic direct summand $\langle a\rangle$ of order $\geq \ord(x)$.
Hence there is a monomorphism $\langle x\rangle\to \langle
a\rangle$ which can be extended to an endomorphism of $G$. It
follows that the heights of all non-zero elements of $\langle
x\rangle$ are finite.

(2) Suppose $G=B\oplus D$ with $\exp(B)=n$   and $D$ divisible.

(a) If $\exp(x)\leq n$ then   as in (1) $G$ has a
cyclic direct summand   of order $\geq \ord(x)$ and
non-zero elements from $\langle x\rangle$ have finite height.

(b) Suppose that $\exp(x)> n$, and $x=b+d$ with $b\in B$ and $d\in
D$. Then $\exp(d)>n\geq \exp(b)$. Let $y\in B$ have exponent $n$.
Then $\exp(y+d)=\exp(d)=\exp(x)$, and there exists $f\in\End( G)$
such that $f(x)=y+d$. Since $h(y)=0$, it follows that $h(x)=0$,
hence $\exp(b)=n$.
\end{proof}

\begin{lemma}\label{sufficient}
If $G$ and $x$ are as in Lemma {\rm \ref{cyclic2}(2)(b)} then
$\langle x\rangle\in\calQ(G)$.
\end{lemma}
\begin{proof}
Since $\exp(b)=\exp(B),\  \langle b\rangle$ is a direct summand of $B$.

Let $f\in\Hom(\la x\ra,G)$ with $f(x)=a+e$, where $a\in B$ and $e\in D$. Since $\la b\ra$ is a summand of $B$ and $\exp(b)\geq\exp(a)$, the map $f_1$ defined by $f_1(b)=a$ extends to $g_1\in \Hom(B,G)$. Since $D$ is injective and $\exp(d)=\exp(x)\geq\exp(e)$, the map $f_2 $ defined by $f_2(d)=e$ extends to $g_2\in\Hom(D,G)$. Hence $g=g_1+g_2$ is an extension of $f$ to $\End(G)$.
\end{proof}

It remains to find intrinsic criteria for
cyclic groups  satisfying the conditions of Lemma \ref{cyclic2}
(1) and (2) (a) to be in $\calQ(G)$. We consider therefore cyclic
groups $\la x\ra$ containing no elements of infinite height in
$G$.

Recall (\cite[Section 65]{fuc2}) that for $n\in\bbN$, an
\textit{Ulm sequence of   length $n$} is a  strictly increasing
infinite  sequence $U=(h_0,h_1,\dots,h_{n-1},\infty,\dots)$ with
each $h_i$ an ordinal,  under the conventions that each ordinal
$h_i<\infty,\ \infty< \infty$ and the constant   sequence
$(\mathbf{\infty})$ is the unique Ulm sequence of length 0. The
set of Ulm sequences   is well--ordered pointwise with maximum
$(\mathbf     \infty)$, no minimum but infimum $\bbN=(0,1,\dots,
n,n+1,\dots)$. This means in particular  that if $U\leq V$ where
$U$ has length $n$ and $V$ has length $m$, then $n\geq m$. An Ulm
sequence $U$ is called \textit{finite} if all its non--infinity
entries are finite.  In particular, $(\infty)$ is a finite Ulm
sequence.

  By Lemma \ref{cyclic2}  we have:
 \begin{corollary}\label{pseudo} If $\la x\ra\in\calP(G)$ then $U(x)$ is finite.
\qed\end{corollary}

We say that the Ulm sequence \textit{$U$ has a gap before   $k$}
if $h_k
>h_{k-1}+1$, where   $h_{-1}$ denotes
by definition the integer $-1$. The gap before   $n$, where $n$ is
the length of $U$, is called \textit{the trivial gap}.

Let $x \in G$ with $\exp(x)=n$. Then $x$ determines an Ulm
sequence of length $n$ by
$U(x)=(h(x),\,h(px),\dots,h(p^{n-1}x),\,\infty,\dots)$. It is
clear from this definition that $U (x)$ is finite if and only if
$\la x\ra\cap G^1=\{0\},\ h(p^kx)=\infty$ if and only if $p^kx\in
D$, the divisible part of $G,\ U(x)=(0,1,\dots,n-1,\infty,\dots)$
if and only if $\la x\ra$ is a summand of exponent $n$ and for
$x,\ y\in G,\ U(x+y)\geq\min\{U(x),\, U(y)\}$.
Finally, note that by \cite[Lemma 65.3]{fuc2}, if $h(x)=0$ and
$U(x)$ has the first non--trivial gap before $k$, then $G$ has a
direct summand of exponent $k$.

By \cite[Lemma 65.5 and Exercise 6]{fuc2} we have:

\begin{lemma}\label{65.5}
Let $G$ be a  group and $x\in G$ such that $\langle
x\rangle\cap G^1=0$.
 \begin{enumerate} \item The following are
equivalent:
 \begin{enumerate} \item $\langle x\rangle\in\calQ(G)$;

\item if $y\in G$ such that $\exp(x)\geq \exp(y)$ then $U(x)\leq
U(y)$. \end{enumerate}

\item The following are equivalent:
 \begin{enumerate} \item $\langle x\rangle\in\calP(G)$;

\item if $y\in G$ such that $\exp(x)= \exp(y)$ then $U(x)\leq
U(y)$.\qed\end{enumerate}
\end{enumerate}
\end{lemma}

Using this result we can characterize cyclic groups $\la x\ra$ with no elements of infinite height in $\calQ(G)$ as follows:

\begin{theorem}\label{3.14}
Let $G$ be a  group and $x\in G$ an element of exponent $n$ such that $\la x\ra\cap G^1=\{0\}$.
The following are equivalent:
\begin{enumerate}
\item $\langle x\rangle\in\calQ(G)$;

\item $\langle x\rangle\in\calP(G)$;

\item $U(x)$ has at most one non-trivial gap and if a gap occurs
before the index $k\geq 0$ and $h(p^{k}x)=k+\ell$,  then $G$ has no
cyclic   summands of exponents between  $k+1$ and $n+\ell-1$.
\end{enumerate}
\end{theorem}

\begin{proof}
(1)$\Rightarrow$(2) This is obvious.

(2)$\Rightarrow$(3) Let $x\in G$ such that $\langle
x\rangle\in\calP(G)$. If $U(x)$ has no non-trivial gaps then
$\langle x\rangle$ is a direct summand of $G$. Therefore we can
assume that $U(x)$ has at least one non-trivial gap.

Suppose that $U(x)=(h_0,\dots,h_{n-1},\infty,\dots)$ has at least
two non-trivial gaps. Since all height $h_i$ are integers, we can
apply \cite[Lemma 65.4]{fuc2}, and it follows that there is a
direct summand $C=\langle c_1\rangle\oplus \dots  \oplus \langle
c_t\rangle$ of $G$ and a strictly increasing chain of positive
integers $0<k_1<k_2<\dots<k_t$ such that
\begin{enumerate} \item[(i)] $t\geq 3$, \item[(ii)]
$\exp(c_1)<\exp(c_2)<\dots<\exp(c_t)=k_t+n$, and \item[(iii)]
$x=p^{k_1}c_1+p^{k_2}c_2+\dots +p^{k_t}c_t$.\end{enumerate}

We observe that the exponent of
$$y=p^{k_1}c_1+p^{k_2-1}c_2+p^{k_3}c_3\dots +p^{k_t}c_t$$ is $n$.
But
$$h(p^{\exp(c_1)-k_1}y)=\exp(c_1)-k_1+k_2-1<\exp(c_1)-k_1+k_2=h(p^{\exp(c_1)-k_1}x),$$
hence $U(y)\ngeq U(x)$, a contradiction.

Therefore $U(x)$ has exactly one non-trivial gap. Let $k$ be the
index such that $U(x)$ has a gap before $k$. Hence
$h(p^kx)=k+\ell$ with $\ell>0$.

Suppose that $\langle z\rangle$ is a direct summand of $G$ of
exponent $n\leq e\leq n+\ell-1$. If $v=p^{e-n} z$ then $p^{k}v\neq
0$ since $n>k$. Moreover,
$$h(p^k v)=e-n+k\leq n+\ell-1-n+k=k+\ell-1<h(p^k x).$$ Therefore
$U(v)\ngeqq U(x)$, but $\exp(v)=\exp(x)$, a contradiction.

Suppose that $\langle z\rangle$ is a direct summand of $G$ of
exponent $k+1\leq e\leq n-1$. We observe that $v=x+z$ is of
exponent $n$. But $h(p^kx)>k=h(p^kz)$, hence
$$h(p^kv)=h(p^kx+p^kz)=k<h(p^kx),$$ and it follows that $U(v)\ngeqq
U(x)$. This leads to a contradiction and the proof is complete.

(3)$\Rightarrow$(1) Let $x$ be as in (3). If $U(x)$ has no
nontrivial gaps then $\langle x\rangle$ is a direct summand of
$G$.

Suppose that $U(x)$ has a gap before the index $k$, and we fix an
element $y$ of exponent $e\leq \exp(x)$. We will prove that
  $U(x)\leq U(y)$.

We consider the Ulm sequence $U(y)=(r_0,\dots,r_{e-1},\infty,\dots)$.

\textbf{Case I:   $r_{e-1}$ is finite.} In
order to prove that $U(x)\leq U(y)$, since $U(x)$ has only one gap
and this occurs before $k$, it is enough to prove that $h(p^k
x)\leq h(p^k y)$.

Suppose by contradiction that $h(p^k x)> h(p^k y)$. As in
\cite[Lemma 65.4]{fuc2}, if $n_1,\dots,n_t$ are the positive
indexes before the gaps occur and we set $r_{n_i}=n_i+k_{i+1}$ and
$k_1=r_0$ then we have cyclic direct summands of exponent
$n_i+k_i$, with $i=1,\dots,t$.

If $k=0$ then we have no cyclic direct summands of exponent
$1,\dots,n+\ell-1$. Then every element $y$ of exponent $\leq n$
must have   height $\geq \ell$.

If $k>0$, let $n_j\leq k$ be the largest index $n_i\leq k$. Then
$h(p^k y)=k+k_{j+1}<k+ \ell$ and $G$ has a direct summand of
exponent $n_{j+1}+k_{j+1}$. Since $k<n_{j+1}\leq n$, we obtain
that $G$ has a cyclic direct summand of exponent $e$ with
$k<e<n+\ell$, a contradiction.

\textbf{Case II: $r_{e-1}$ is infinite.} Let $u=h(p^{n-1}x)$. If
$B=\bigoplus_{i>0}B_i$ is a basic subgroup of $G$, where $B_i$ are
homogeneous subgroups of exponent $i$, we consider the direct
decomposition $G=B_1\oplus\dots \oplus B_u\oplus p^u G$ and write
$y=y_1+\dots +y_u+y^*$ with $y_i\in B_i$ for all $i=1,\dots,u$ and
$y^*\in p^u G$. It is obvious that $U(x)\leq U(y^*)$. Moreover,
$y-y^*$ satisfies Case I and $\exp(y-y^*)<\exp(x)$. Therefore
$U(x)\leq\min\{U(y-y^*),U(y^*)\}$, and it follows that $U(x)\leq
U(y)$.
\end{proof}

\begin{corollary} Let $\la  x\ra\in\calQ(G)$. If $U(x)$ has no non--trivial gap then $\la x\ra\in\calS(G)$. If $U(x)$ has a non--trivial gap at   $k$, then $x\in H$,  a summand of $G$, where $H$ is cyclic if $k=0$  and finite of rank $2$ otherwise.
\qed\end{corollary}

}

It is easy to see that $\calQ(G)$ is closed with respect direct
summands.

\begin{corollary} The set  $\calQ(G)$ is not closed under direct sums, even in the case that $G$ is a finite $p$--group.
\end{corollary}
\begin{proof} Let $x,\ y\in G$ such that $U(x)$ has a single non--trivial gap before index $k$ and $h(p^kx)=k+\ell$ wih $\ell>1$. Let $\la y\ra$ be cyclic of exponent $k+1$. Then $\la x\ra$ and $\la y\ra\in\calQ(G)$ but $\la x\ra\oplus\la y\ra\not\in\calQ(G)$.
\end{proof}

\section{Finite subgroups in $\calQ(G)$}\label{finite}

We now extend the results of Section \ref{sect-cyclic} from cyclic
to finite subgroups.   The main result of this section is that a
finite subgroup $H$ of a group $G$ is in $\calQ(G)$ if and
only if it is a valuated direct sum of cyclic subgroups from
$\calQ(G)$.

Recall from \cite{HRW77} that if $H\subseteq G$,   the
\textit{ valuation of $H$ induced by heights in $G$} is defined by
$v(x)=h(x)$, the height of $x$ in $G$, for all $x\in H$ and
$H=K\oplus L$  is a \textit{valuated direct sum} if
$v(k+\ell)=\min\{v(k), v(\ell)\}$ for all $k\in K$ and $\ell\in
L$. In the following results, the valuation of $H$ is always that
induced by heights in $G$. Consequently, $f\in\Hom(H,G)$ does not
decrease valuations if and only if $f$ does not decrease Ulm
sequences in $G$.

 \begin{lemma} \label{valuated} Let $G$ be a
 group and let $K,\ L\leq \calQ(G) $   with $K\cap L=0$. If
$K\oplus L\in\calW(G)$ then $K\oplus L$ is a valuated direct sum.
\end{lemma}

\begin{proof}  Suppose the direct sum $K\oplus L$ is not valuated. Then there exists $(k,\ell)\in K\oplus L$ such that $h^G((k,\ell))>\min\{h^G(k),\,h^G(\ell)\}$. For example, say $h^G((k,\ell))>h^G(k)$. Let $f\in\End(K\oplus L)$ be the natural projection onto $K$. Then $h^G(f(k,\ell))=h^G(k)<h^G((k,\ell))$ so $f$ cannot be extended to $\End(G)$, a contradiction.
 \end{proof}

\begin{lemma}\label{finite-1}
Let $G$ be a  group, $K$  a {pure} subgroup of $G$ such
that $K$ is a  direct sum of cyclic groups
  and $G/K$ is divisible, and $H\leq K$   a finite
subgroup. If $f:\,H\to G$ is a homomorphism, the following are
equivalent:
\begin{enumerate}
\item $f$ can be extended to an endomorphism of $G$;

\item  $f$ can be extended to a homomorphism $\overline{f}:K\to
G$;

\item $f$ can be extended to a homomorphism $\overline{f}:K\to G$
such that $\overline{f}(K)$ is bounded.
\end{enumerate}
\end{lemma}

\begin{proof}
(1)$\Rightarrow$(2) is obvious.

(2)$\Rightarrow$(3) We extend $f$ to a homomorphism $g:K\to G$.
Since $H$ is finite and $K$ is a   direct sum of cyclic groups,
there is a finite direct summand $L$ of $K$ such that $H\leq L$.
If $L\oplus M=K$, we define $\overline{f}:K\to G$ by
$\overline{f}(x+y)=g(x)$ for all $x\in L$ and $y\in M$.

(3)$\Rightarrow$(1) Let $\overline{f}$ be as in (3). We have to
extend $\overline{f}$ to an endomorphism of $G$. Let $k>0$ be an
integer such that $\overline{f}(K)$ is bounded by $p^k$. Since
$G/K$ is divisible, it is easy to see that $G=K+p^kG$, hence for
every $x\in G$ there are $y\in K$ and $z\in G$ such that
$x=y+p^kz$. {Since $K$ is pure in $G$}, it is not hard to
see that the map $g:G\to G$, $g(x)=\overline{f}(y)$ is well
defined, and it represents an endomorphism of $G$ which extends
$f$.
\end{proof}

\begin{corollary}\label{finite-1-Q-W}
 Let $G$ be a group and let $K$ be a pure subgroup of $G$ such that $K$ is a
direct sum of cyclic groups
with $G/K$ is divisible. Let   $H$ be  a finite
subgroup of $K$.
 \begin{enumerate}\item The following are equivalent:
\begin{itemize} \item[(a)] $H\in\calQ(G)$;
\item[(b)]  every homomorphism $f:H\to G$ can be extended to a
homomorphism $\overline{f}:K\to G$; \item[(c)] every homomorphism
$f:H\to G$ can be extended to a homomorphism $\overline{f}:K\to G$
such that $\overline{f}(K)$ is bounded.
\end{itemize}
\item The following are equivalent:
\begin{itemize} \item[(a)] $H\in\calW(G)$;
\item[(b)]  every endomorphism $f:H\to H$ can be extended to a
homomorphism $\overline{f}:K\to G$; \item[(c)] every endomorphism
$f:H\to H$ can be extended to a homomorphism $\overline{f}:K\to G$
such that $\overline{f}(K)$ is bounded.
\qed\end{itemize}
\end{enumerate}
\end{corollary}

We are now able to characterize finite  subgroups of $G$ in
$\calQ(G)$, respectively in $\calW(G)$.

\begin{theorem}\label{finite-Q(G)}
Let $G$ be a group and let $H=\bigoplus_{i=1}^n H_i$ be a
finite subgroup such that all $H_i$ are cyclic groups. The
following are equivalent:

\begin{enumerate}
\item $H\in\calQ(G)$;

\item \begin{enumerate} \item $H_i\in\calQ(G)$ for all
$i=1,\dots,n$,

\item $H=\bigoplus_{i=1}^n H_i$ is a   valuated direct sum of
cyclic groups.
\end{enumerate}
\end{enumerate}
\end{theorem}

\begin{proof}
(1)$\Rightarrow$(2) This follows from Lemma \ref{valuated}.

(2)$\Rightarrow$(1) \textbf{Case I: $G$ has an unbounded basic
subgroup.} Then $H_i\cap G^1=0$ for every $i\in \{1,\dots,n\}$.
Since the direct sum $\bigoplus_{i=1}^n H_i$ is valuated, it
follows that $H\cap G^1=0$, hence there is a basic subgroup $B\leq
G$ such that $H\leq B$. By Lemma \ref{finite-1}, it is enough to
prove that every homomorphism $f:H\to G$ can be extended to a
homomorphism $f':B\to G$.

We consider a homomorphism $f:H\to G$. If $x\in B$ we   denote by
$h^B(x)$ the height of $x$ calculated in $B$ and by $h(x)$ the
height of $x$ as an element of $G$.

Then the restrictions $f|_{H_i}$ can be extended to endomorphisms
of $G$, and it follows that $h(x_i)\leq h(f(x_i))$ for all $i$ and
all $x_i\in H_i$.

Let $x=x_1+\dots+x_n\in H$ with $x_i\in H_i$ for all $i$. We
observe that \begin{align*} h^B(x)\leq h(x)&=\min\{h(x_i)\mid
i=1,\dots,n\}\leq \min\{h(f(x_i))\mid i=1,\dots,n\}\\ &\leq
h(f(x_1)+\dots+f(x_n))=h(f(x)).\end{align*} Since $B/H$ is a
direct sum of cyclic groups, and $H$ is a nice subgroup of $B$ as
a consequence of \cite[79(b)]{fuc2}, we can apply \cite[Corollary
81.4]{fuc2} to conclude that there is a homomorphism $f':B\to G$
which extends $f$, and the proof is complete.

\textbf{Case II: $G=B\oplus D$ with $B$ bounded and $D$
divisible.} Let $X=\{i\in\{1,\dots,n\}\mid H_i\cap G^1=0\}$ and
$Y=\{1,\dots,n\}\setminus X$. Since the direct sum
$\bigoplus_{i=1}^n H_i$ is valuated, it follows that
$\bigoplus_{i\in X} H_i\cap G^1=0$, hence we can suppose
$\bigoplus_{i\in X} H_i\leq B$.

For every $i\in Y$ we fix a generator $h_i$ for $H_i$, and
write $h_i=b_i+d_i$ with $\langle b_i\rangle$ a direct summand of
$B$, $d_i\in D$. Since $H_i\cap G^1\neq 0$, it follows that
$\exp(B)=\exp(b_i)<\exp(d_i)=\exp(h_i)$ by Lemma \ref{cyclic2}.

We claim that $\sum_{i\in Y}\langle b_i\rangle=\bigoplus_{i\in
Y}\langle b_i\rangle$. In order to prove this, suppose by
contradiction that there exist an index $j\in Y$ and a non-zero
element $0\neq k_j b_j=\sum_{i\in Y\setminus \{j\}}k_i b_i$. Then
$k_j h_j- \sum_{i\in Y\setminus \{j\}}k_i h_i$ is of infinite
height, hence the sum $\bigoplus_{i\in Y} H_i$ is not valuated
since $k_jh_j$ is of finite height. This contradicts our
hypothesis, so the claim is true. Moreover, $\bigoplus_{i\in
Y}\langle b_i\rangle$ is a direct summand of $B$ as a bounded pure
subgroup, and using \cite[Exercise 9.8]{fuc1} we conclude that it
is an absolute direct summand of $B$.

Using a similar argument, if we suppose that $(\bigoplus_{i\in
Y}\langle b_i\rangle)\cap (\bigoplus_{i\in X}H_i)\neq 0$ we obtain
that $\bigoplus_{i=1}^{n}H_i$ is not a valuated direct sum, a
contradiction. Hence $$\textstyle{(\bigoplus_{i\in Y}\langle
b_i\rangle)\cap (\bigoplus_{i\in X}H_i) =0,}$$ and it follows that
there is a direct complement $C$ of $(\bigoplus_{i\in Y}\langle
b_i\rangle)$ in $B$ such that $\bigoplus_{i\in X}H_i\leq C$.

Moreover, since the sum $\bigoplus_{i\in Y}H_i$ is direct and
$H_i[p]=\langle d_i\rangle[p]$ for all $i\in Y$, it follows that
the sum $\sum_{i\in Y}\langle d_i\rangle$ is a direct sum. Hence
we can find infinite quasi-cyclic subgroups $D_i$, $i\in Y$, such
that $D=(\bigoplus_{i\in Y}D_i)\oplus D'$ and $d_i\in D_i$ for all
$i\in Y$.

Let $f:H\to G$ be a homomorphism. Using the same argument as in
the first case we observe that every homomorphism $\bigoplus_{i\in
X} H_i\to G$ can be extended to a homomorphism $f_0:C\to G$ (note
that the valuation induced on $\bigoplus_{i\in X} H_i$ by $C$ is
the same as the valuation induced by $G$).

For all $i\in I$, we have $\exp(f(h_i))\leq\exp(h_i)$, hence
$f(h_i)=a_i+z_i$ with $a_i\in B$ and $z_i\in D$ such that
$\exp(z_i)\leq\exp(d_i)$. Therefore there exist homomorphisms
$f'_i:\langle b_i\rangle \to G$ such that $f'_i(b_i)=a_i$ and
$f''_i:D_i\to D$ such that $f''_i(d_i)=z_i$. %Therefore $(f_1+
%f_2)(x)=y$.

Since $$\textstyle{G=C\oplus (\bigoplus_{i\in Y}\langle
b_i\rangle)\oplus (\bigoplus_{i\in Y}D_i)\oplus D',}$$ the
homomorphisms $f_0$, $f'_i$ and $f''_i$, $i\in Y$, induce an
endomorphism $\overline{f}:G\to G$, and it is easy to see that
$\overline{f}$ extends $f$.
\end{proof}

\begin{remark}
Cyclic valuated groups are characterized using invariants in
\cite[Theorem 3]{HRW77}. Therefore this result together with
Theorem \ref{finite-Q(G)} and Theorem \ref{3.14} give us a
characterization by invariants for subgroups in $\calQ(G)$.
\end{remark}

We close this section with a characterization of
some finite subgroups in $\calW(G)$.

\begin{theorem}\label{finite-W(G)}
Let $G$ be a group and $H=\bigoplus_{i=1}^n H_i$ a finite subgroup
such that   $H\cap p^\omega G=0$ and each $H_i=\langle z_i\rangle$
is a cyclic group of exponent $e_i$. The following are equivalent:

\begin{enumerate}
\item $H\in\calW(G)$;

\item \begin{enumerate} \item If $e_j\leq e_i$ then $U(z_i)\leq
U(z_j)\leq U(p^{e_i-e_j}z_i)$,

\item $H=\bigoplus_{i=1}^n H_i$ is a   valuated direct sum of
cyclic groups.
\end{enumerate}
\end{enumerate}
\end{theorem}

\begin{proof}
(1)$\Rightarrow$(2) In order to prove (a), let $i,j$ be two
  indices such that $e_j\leq e_i$. Then there are homomorphisms
$f:H_i\to H_j$ with $f(z_i)=z_j$ and $g:H_j\to H_i$ with
$f(z_j)=p^{e_i-e_j}z_i$. Since these homomorphisms can be extended
to endomorphisms of $H$, they can be extended to endomorphisms of
$G$. The inequalities $U(z_i)\leq U(z_j)\leq U(p^{e_i-e_j}z_i)$
follow from the fact that   endomorphisms do  not decrease
heights.

The statement (b) is   a consequence of Lemma \ref{valuated}.

(2)$\Rightarrow$(1) As in the proof of Theorem \ref{finite-Q(G)},
it is enough to prove that every homomorphism of $f:H_i\to H$ does
not decrease the valuation.

Let $f:H_j\to H$ be a homomorphism defined for some
$j\in\{1,\dots,n\}$. Since $U(mz)=U(z)$ for all integers $m$ with
$(m,p)=1$, it is enough to prove that $U(p^kz_j)\leq
U(p^k(f(z_j)))$ for all $0\leq k<e_j$. Since for every element $x$
and for every positive integer $k$ the indicator $U(p^kx)$ can be
obtained by deleting the first $k$ components of $U(x)$, it is
enough to prove $U(z_j)\leq U(f(z_j))$.

Let $f(z_j)=\sum_{i=1}^{n}m_iz_i$. Note that if $e_j< e_i$ then
$p^{e_i-e_j}$ divides $m_i$. Then
$$\textstyle{ f(z_j)=(\sum_{e_j<e_i}n_ip^{e_i-e_j}z_i)+(\sum_{e_i\leq
e_j}m_iz_i),}$$ hence $$U(f(z_j))=\min \{U(n_ip^{e_i-e_j}z_i)\mid
e_j<e_i\}\cup\{U(m_i z_i)\mid e_i\leq e_j\}\geq U(z_j),$$ and the
proof is complete.
\end{proof}

In the following example we show that both pairs of conditions (a)
and (b) in Theorem \ref{finite-Q(G)} and Theorem \ref{finite-W(G)}
respectively are necessary.

\begin{example}\label{example-nonQ-valuated}
Let $G=\langle a\rangle\oplus \langle b\rangle\oplus \langle
c\rangle$ be a group with $\exp(a)=1$ and $\exp(b)=\exp(c)=2$.
Then $x=(a,pb,0)$ and $y=(a,0,pc)$ generate direct summands. We
have $\langle x,y\rangle=\langle x\rangle\oplus \langle
y\rangle=\langle x\rangle\oplus \langle (0,-pb,pc)\rangle$. The
direct sum $\langle x\rangle\oplus \langle y\rangle$ is not a valuated
direct sum, while $\langle x\rangle\oplus
\langle (0,-pb,pc)\rangle$ is a valuated direct sum, but
$\langle (0,-pb,pc)\rangle\notin \calQ(G)$.
\end{example}

 \begin{remark}\label{counterexample}  The result of Theorem \ref{finite-W(G)} cannot be extended to infinite direct sums of cyclic groups.
Pierce \cite[Theorem 15.4]{P63} has constructed an example of a
separable $p$--group $G$ with standard basic subgroup (i.e.
$B=\bigoplus_{i=1}^\infty \bbZ(p^i))$ such that $\End(G)=J +E $
where $J$ is the rank 1 torsion--free complete $p$--adic module
generated by the identity and $E$ is the ideal of small
endomorphisms. On the other hand, $\End(B)$ has infinite
torsion--free $p$--adic rank, so $B\not\in\calW(G)$.
\end{remark}

\section{Finitely generated subgroups in $\calP(G)$}

Jain and Singh proved in \cite{JS67} that  if $R$ is a principal
ideal domain, then all pseudo--injective modules are
quasi--injective. In this section we prove a stronger version of
this for $R=\bbZ$:  if $G$ is a group then all   finitely
generated subgroups   in $\calP(G)$   are in $\calQ(G)$.

\textsl{In this section, $G$ is an arbitrary abelian group.}

  In order to prove $\calQ_f(G)=\calP_f(G)$ for
  all groups $G$ we start with the
  case of $p$-groups.

\begin{lemma}
Let $G$ be a $p$-group and $H\in\calP(G)$. If $K$ is a cyclic
direct summand of $H$ then $K\in\calP(G)$.
\end{lemma}

\begin{proof}
Let $L$ be a direct complement of $K$ in $H$, so $H=K\oplus L$,
and let $\varphi:K\to G$ be a monomorphism.

If $\varphi(K)\cap L=0$, then the homomorphism $\psi:K\oplus L\to
G$, $\psi(x,y)=\varphi(x)+y$ is a monomorphism, hence it can be
extended to an endomorphism $\overline{\varphi}\in\End(G)$. It is
easy to see that $\overline{\varphi}$ also extends $\varphi$.

If $\varphi(K)\cap L\neq 0$, we first observe that the socle
$\varphi(K)[p]$ is contained in $L$ since $\varphi(K)$ is a cyclic
$p$-group (hence   its subgroup lattice is a finite chain).
Let $\varphi':K\to G$ be the homomorphism defined by
$\varphi'(x)=\varphi(x)-x$. Suppose that $\varphi'$ is not a
monomorphism. Then there exists a non-zero element $x\in K$ such
that $\varphi'(x)=0$. Then $\varphi(x)=x\in K$, hence
$\varphi(K)\cap K\neq 0$. It follows that $\varphi(K)[p]\subseteq
K$, and this contradicts $K\cap L=0$. Hence $\varphi'$ is a
monomorphism. Suppose that $\varphi'(K)\cap L\neq 0$. Then there
exists $x\in K$ such that $0\neq \varphi(x)-x\in L$. If $e$ is the
exponent of $x$ then $p^{e-1}(\varphi(x)-x)\in L[p]$. But
$p^{e-1}\varphi(x)\in\varphi(K)[p]\subseteq L[p]$, hence
$p^{e-1}x\in L[p]$, a contradiction. Then $\varphi'(K)\cap L=0$
and we can apply what we proved so far to observe that there
exists an endomorphism $\psi$ of $G$ which extends $\varphi$. Then
for every $x\in K$ we have $\varphi(x)=\psi(x)+x$, hence
$\psi+1_G$ extends $\phi$.
\end{proof}

We need the following technical result:

\begin{lemma}\label{lemma-dsvc}
Let $G$ be a $p$-group and $H=\bigoplus_{i=1}^n H_i$ a finite
subgroup such that all $H_i=\langle h_i\rangle$ are cyclic groups
such that
\begin{enumerate}
\item[(i)] $\exp(h_1)\leq \exp(h_2)\leq\dots\leq\exp(h_n)$,

\item[(ii)] for all $m\in\{1,\dots,n\}$ and for all $x\in
\bigoplus_{i=1}^m H_i$ we have $U(h_m)\leq U(x)$, and

\item[(iii)] if $\exp(h_i)=\exp(h_j)$ then $U(h_i)=U(h_j)$.
\end{enumerate}
Then the direct sum $\bigoplus_{i=1}^n H_i$ is a valuated direct
sum of cyclic $p$-groups.
\end{lemma}

\begin{proof}
The proof is by induction. For $n=1$ the property is obvious.
Suppose that (ii) is valid for all $m<n$. Then
$\bigoplus_{i=1}^{n-1} H_i$ is a valuated direct sum of cyclic
groups. Let $k$ be the minimal index such that
$\exp(H_n)=\exp(H_k)$.

We observe that the sequence $U(h_i),\ i=1,\dots,n$ is a
decreasing sequence such that $U(h_i)=U(h_j)$ if and only if
$\exp(H_i)=\exp(H_j)$. Moreover, it follows by (b) that
$U(h_n)\leq U(y)$ for all $y\in H$.

If $U$ is an Ulm sequence, we denote, as in \cite[p.100]{HRW77},
$$H(U)=\{x\in H:\, U(x)\geq U\},$$ $$H(U)^*=\{x\in H:\,
U(x)> U\}$$ and we consider the $\Z(p)$-vector space
$$H_U=\frac{H(U)+pH}{H(U    )^*+pH}.$$ Recall that a
\textsl{$v$-basis} for $H$ is constructed in the following way:
for every $U    $ we fix a basis in $H_U    $, and we choose one
representative whose Ulm sequence is $U    $ for each element in
this basis; the union of all these representatives is a $v$--basis
for $H$. It is proved in \cite[Theorem 3]{HRW77} that $H$ is a
valuated direct sum of cyclics  if and only if the cardinal of a
$v$--basis coincides to the rank of $H$. Moreover, in this
hypothesis every $v$--basis is linearly independent and it
generates $H$. Therefore, it is enough to prove that $\{h_i:\,
i=1,\dots,n\}$ is a $v$--basis for $H$.

Since $K=\bigoplus_{i=1}^{n-1} H_i$ is a direct sum of cyclic
valuated groups, it follows that the set $\{h_i:\,
i=1,\dots,n-1\}$ is a $v$--basis for $K$.

Let $V    $ be the Ulm sequence      of $h_n$. If $U< V    $ then
$H(U)=H(U    )^*=H$, so $H_U    =0$. If $U    $ and $V    $ are
not comparable then $H(U    )=H(U    )^*$ since $V    $ is minimal
as Ulm sequence      of an element of $H$, so $H_U    =0$.

It is easy to see that $H(V    )=H$, and $H(V
)^*=(\bigoplus_{i<k}H_i)\oplus(\bigoplus_{i=k}^n pH_i)$, so
$h_{k},\dots,h_{n}$ represent a basis in $H_V    $.

If $V    <U    $ then $$\textstyle{H(U )+pH=(\bigoplus_{U(h_i)\geq
U }H_i)+pH=K(U    )\oplus pH_n}$$ and
$$\textstyle{H(U    )^*+pH=(\bigoplus_{U(h_i)> U    }H_i)+pH=K(U    )^*\oplus
pH_n.}$$ Therefore every set in $K=\bigoplus_{i=1}^{n-1} H_i$
which represents a basis in $K_U    $ is also a representative set
for a basis in $H_U    $.

It follows that $\{h_1,\dots,h_n\}$ is a $v$-basis, and an
application of the proof of \cite[Theorem 3]{HRW77} will complete
the proof.
\end{proof}

\begin{lemma}\label{basic-in}
Let $G$ be a $p$-group and $x,y\in G$. If $U(x+y)=U(x)$ then
$U(x)\leq U(y)$.
\end{lemma}

\begin{proof}
Suppose that $U(x)\nleqq U(y)$. Then there exists a positive
integer $k$ such that $h(p^k y)<h(p^k x)$. It follows that
$h(p^k(x+y))=h(p^k y)\neq h(p^k x)$, and this contradicts our
hypothesis.
\end{proof}

\begin{theorem}\label{Pf=Qf}
Let $G$ be a group. Then $\calP_f(G)=\calQ_f(G)$.
\end{theorem}

\begin{proof} Since only the inclusion $\calP_f(G)\subseteq \calQ_f(G)$ requires a proof, we
start with a finitely generated subgroup $H\in\calP(G)$.

If $G$ is a $p$-group, we write $H=\bigoplus_{i=1}^n H_i$ such
that all $H_i=\langle h_i\rangle$ are cyclic groups with
$$\exp(h_1)\leq \exp(h_2)\leq\dots\leq\exp(h_n).$$ We will prove
that this decomposition satisfies the conditions (ii) and (iii)
from Lemma \ref{lemma-dsvc}.

Let $m\in\{1,\dots,n\}$ and $j<m$. Then $H=\bigoplus_{i=1}^n
H'_i$, where $H'_i=\langle h_i\rangle$ for all $i\neq m$ and
$H'_i=\langle h_m+h_j\rangle$. Since $H\in\calP(G)$, the
isomorphism $\varphi:H\to H$ defined by $\varphi(h_i)=h_i$ for all
$i\neq m$ and $\varphi(h_m)=h_m+h_j$ can be extended to an
endomorphism of $G$. Then $U(h_m)\leq U(h_m+h_j)$. But
$\varphi^{-1}$ also can be extended to an endomorphism of $G$,
hence $U(h_m)\geq U(h_m+h_j)$. Therefore, $U(h_m)= U(h_m+h_j)$,
and applying Lemma \ref{basic-in} we obtain that the condition
(ii) is satisfied.

In order to prove (iii), it is enough to observe that if
$\exp(h_i)=\exp(h_j)$ with $i<j$ we can replace in the direct
decomposition of $H$ as direct sum of cyclic groups either of the
two direct summands $\langle h_i\rangle$ or $\langle h_j\rangle$
by $\langle h_i+h_j\rangle$. By what we just proved for (ii) we
have $U(h_i)=U(h_j)$.

Therefore, we proved that $\calQ_f(G)=\calP_f(G)$ for all
$p$-groups. It is not hard to extend this property to all torsion
groups and we now show that result can be extended to all abelian
groups.

Let $G$ be a group and $H\in \calP_f(G)$. Since $H$ is finitely
generated, $H=F\oplus K$, with $F$ a free subgroup of finite rank
and $K$ a finite subgroup. We claim that $F$ and $K$ are in
$\calQ(G)$ and every homomorphism $\varphi:F\to G$ can be extended
to an endomorphism $\overline{\varphi}$ of $G$ such that
$\overline{\varphi}(K)=0$.

Since $K\leq T(G)$, every monomorphism $\varphi:K\to G$ can be
extended to a monomorphism $\varphi':H\to G$ such that
$\varphi'(x)=x$ for all $x\in F$. Then there is an endomorphism
$\overline{\varphi}$ of $G$ which extends $\varphi'$. Since $T(G)$
is fully invariant, $\overline{\varphi}$ induces an endomorphism
of $G$ which extends $\varphi$. Therefore
$K\in\calP_f(T(G))=\calQ_f(T(G))$, and it follows that we can embed
$K$ in a finite direct summand $L$ of $G$.

In order to prove $F\in\calQ(G)$, let $\varphi:F\to G$ be a
homomorphism. For every positive integer $i$ we consider the
subgroup $$\textstyle{U_i=\{x\in F\mid \varphi(x)=ix\}\leq F,}$$
and we will prove by induction on $n$ that
$$\textstyle{\sum_{i=1}^nU_i=\bigoplus_{i=1}^nU_i}$$ for all $n>0$. Since the case
$n=1$ is obvious, suppose that
$\sum_{i=1}^nU_i=\bigoplus_{i=1}^nU_i$. Let $x\in
(\sum_{i=1}^nU_i)\cap U_{n+1}$. Then $x=\sum_{i=1}^nx_i$ with
$x_i\in U_i$, hence
$$\textstyle{(n+1)\sum_{i=1}^nx_i=\varphi(x)=\sum_{i=1}^n\varphi(x_i)=\sum_{i=1}^nix_i.}$$
Then $\sum_{i=1}^n(n+1-i)x_i=0$, and by the induction hypothesis
$(n+1-i)x_i=0$ for all $i=1,\dots,n$. Since $F$ is torsion free,
it follows that $x=0$. Then $\sum_{i>0}U_i=\bigoplus_{i>0}U_i\leq
F$. But $F$ is of finite rank, hence we can find an integer $N>0$
such that $U_n=0$ for all $n\geq N$. Let $q>N$ be a prime such
that $\gcd(q,|K|)=0$. Therefore the homomorphism $\psi:F\to G$,
$\psi(x)=\varphi(x)-qx$ is a monomorphism. Then $\psi(F)$ is
torsion--free, and as in the first part of the proof it can be
extended to a monomorphism $\psi':H\to G$ such that $\psi'(x)=-qx$
for all $x\in K$. Since $H\in\calP(G)$ there is an endomorphism
$\overline{\psi}$ of $G$ which extends $\psi'$. Then
$\overline{\varphi}=\overline{\psi}+q1_G$ is an endomorphism of
$G$ which extends $\varphi$ and $\overline{\varphi}(K)=0$.

Now we will prove that every homomorphism $\varphi:K\to G$ can be
extended to an endomorphism $\overline{\varphi}$ of $G$ such that
$\overline{\varphi}(F)=0$. Let $\varphi:K\to G$ be a homomorphism.
If $\varphi':G\to G$ extends $\varphi$ then the restriction
$\varphi'|_{F}:F\to G$ can be extended to an endomorphism $\psi$
of $G$ such that $\psi(K)=0$. Then
$\overline{\varphi}=\varphi'-\psi$ has the required properties.

In order to complete the proof, let $\varphi:\,H\to G$ be a
homomorphism. Then it induces by restrictions two homomorphisms
$\varphi_1:F\to G$ and $\varphi_2:K\to G$. But by what we just proved
we can extend these homomorphisms to two endomorphisms
$\overline{\varphi}_1$ and $\overline{\varphi}_2$ of $G$ such that
$\overline{\varphi}_1$ extends $\varphi_1$ and
$\overline{\varphi}_1(K)=0$, and $\overline{\varphi}_2$
extends ${\varphi}_2$ and $\overline{\varphi}_2(F)=0$. Then
$\overline{\varphi}=\overline{\varphi}_1+\overline{\varphi}_2$
extends $\varphi$ and the proof is complete.
\end{proof}

\begin{remark}
In the case $G$ is a $p$-group with $p\geq 3$ there is a   simpler
proof. In fact, in order to prove $H\in \calQ(G)$ for all finite
subgroups $H\in\calQ(G)$ it is enough to prove that Lemma
\ref{valuated} is valid for $K\oplus L\in\calP(G)$. For the case
$p\geq 3$ it is easy to see that the homomorphism $\varphi:K\to
L\to G$, $\varphi(k,\ell)=2k+\ell$ is a monomorphism, hence it can
be extended to an endomorphism $\overline{\varphi}$ of $G$. Then
$\overline{\varphi}-1_G$ extends the canonical projection $K\oplus
L\to K$ and the proof presented for Lemma \ref{valuated} also
works for $\calP(G)$.
\end{remark}

A careful analysis of the previous proof reveals the fact that $F$
can be replaced by any finite rank torsion-free group. Therefore
we obtain:

\begin{corollary}
 Let $G$ be a group and $H$ a  subgroup of
$G$ of finite rank. Then $H\in\calQ(G)$ if and only if $H\in\calP(G)$.
\end{corollary}

We have been unable to determine whether   Theorem \ref{Pf=Qf} can
be extended to all subgroups in $\calP(G)$, i.e. if there exists
an abelian group $G$ such that $\calQ(G)\neq \calP(G)$.

\end{document}